\documentclass[12pt]{article}

\usepackage{amsthm,amsmath,amssymb}

\usepackage{graphicx}

\usepackage[colorlinks=true,citecolor=black,linkcolor=black,urlcolor=blue]{hyperref}

\theoremstyle{plain}
\newtheorem{theorem}{Theorem}
\newtheorem{lemma}[theorem]{Lemma}
\newtheorem{corollary}[theorem]{Corollary}
\newtheorem{proposition}[theorem]{Proposition}

\newtheorem{observation}[theorem]{Observation}

\theoremstyle{definition}

\newtheorem{example}[theorem]{Example}

\theoremstyle{remark}

\newcommand{\Z}{\mbox{${\mathbb
 Z}$}}

\title{\bf The structure of rainbow-free colorings for linear equations on three variables in $\Z_p$}

\author{Mario Huicochea\\
\small CINNMA\\[-0.8ex]
\small San Isidro 303 Juriquilla Fracc. Villas del Mes\'{o}n\\[-0.8ex] 
\small\tt dym@cimat.mx\\
\and
Amanda Montejano \thanks{Research supported by CONACyT-Project 166306, and PAPIIT IA102013}\\
\small Universidad Nacional Aut\'{o}noma de M\'{e}xico\\[-0.8ex]
\small Facultad de Ciencias UMDI--Juriquilla\\[-0.8ex]
\small Boulevard Juriquilla No. 3001 Juriquilla, Quer\'{e}taro 76230, M\'{e}xico\\[-0.8ex]
\small\tt amandamontejano@ciencias.unam.mx
}

\date{}

\begin{document}

\maketitle

% E-JC papers must include an abstract. The abstract should consist of a
% succinct statement of background followed by a listing of the
% principal new results that are to be found in the paper. The abstract
% should be informative, clear, and as complete as possible. Phrases
% like "we investigate..." or "we study..." should be kept to a minimum
% in favor of "we prove that..."  or "we show that...".  Do not
% include equation numbers, unexpanded citations (such as "[23]"), or
% any other references to things in the paper that are not defined in
% the abstract. The abstract will be distributed without the rest of the
% paper so it must be entirely self-contained.

\begin{abstract}

Let $p$ be a prime number and $\Z_p$ be the cyclic group of order $p$.  A coloring of $\Z_p$ is called rainbow--free with respect to a certain equation, if it contains no rainbow solution of the same, that is, a solution whose elements have pairwise distinct colors. In this paper we describe the structure of rainbow--free $3$--colorings of $\Z_p$ with respect to all linear equations on three variables.  
Consequently, we determine those linear equations on three variables for which every $3$--coloring (with nonempty color classes) of $\Z_p$ contains a rainbow solution of it. 

  % keywords are optional
\end{abstract}

\section{Introduction}

A $k$--\emph{coloring} of a set $X$ is a surjective mapping  $c : X \rightarrow \{1,2,...k\}$, or equivalently  a partition $X=C_1\cup C_2\cup... \cup C_k$, where each nonempty set $C_i$ is called a \emph{color class}. A subset $Y \subseteq X$ is \emph{rainbow} under $c$, if the coloring pairwise assigns distinct colors to the elements of $Y$. The study of the existence of rainbow structures falls into the anti--Ramsey Theory. Canonical versions of this theory prove the existence of either a monochromatic structure, or a rainbow structure. In contrast, in the recent so--called \emph{Rainbow Ramsey Theory} the existence of rainbow structures is guaranteed under some density conditions on the color classes (see \cite{axe,jetall,jnr} and references therein). Beyond this approach, recent works  \cite{llm1, ms} have addressed the problem of describing the shape of colorings containing no rainbow structures, called \emph{rainbow--free} colorings. 

Let $p$ be a prime number and $\Z_p$ be the cyclic group of order $p$. Among other results, Jungi\'{c} et al. \cite{jetall} proved  that every $3$--coloring of $\Z_p$ with the cardinality of the smallest color class greater than four has a rainbow solution of all linear equations in three variables with the only possible exception of $x+y+z=d$. In other words, the authors proved that rainbow--free colorings of $\Z_p$ concerning the equation $a_1x+a_2y+a_3z=b$, where some  $a_i\neq a_j$, are such that the smallest color class has less than four elements. In this work we analyze the ``small cases'' (cases when the smallest color class has one, two or three elements) in order to fully characterize the structure of rainbow--free colorings.

Our main result, Theorem~\ref{thm:main}, implies that, actually, rainbow--free colorings of $\Z_p$ concerning equation  $a_1x+a_2y+a_3z=b$, with some $a_i\neq a_j$, are such that the cardinality of the smallest color class is one. Moreover, Theorem~\ref{thm:main} characterizes the structure of such  colorings. Therefore, we  provide a criterion to decide whether or not, for a given equation and a given prime number, there exists  a rainbow--free coloring. In other words, we classify equations (depending on $a_1,a_2,a_3,b$ and $p$) for which every $3$--coloring contains a rainbow solution  (Corollary~\ref{thm:main2}).   

The paper is organized as follows: In Section~\ref{sec:pre} we establish the notation  and give some preliminary results. In Section~\ref{sec:result} we present our results: First we give the structure characterization of rainbow free colorings of $\Z_p$ concerning equation $x+y+z=b$ (Theorem~\ref{thm:a1=a2=a3}),  which is the only one admitting rainbow--free colorings with large color classes. We point out that Theorem~\ref{thm:a1=a2=a3} is deduced  from known results with relatively little effort. In Theorem~\ref{thm:main} we  give the structure characterization of rainbow--free colorings concerning equation $a_1x+a_2y+a_3z=b$ with some $a_i\neq a_j$. The proof of Theorem~\ref{thm:main} is divided into three parts. In Section~\ref{sec:A=1} we handle the case  when there is a color class of cardinality one, which is the only case where rainbow--free colorings exist. In Sections~\ref{sec:A=2} and \ref{sec:A=3} we discard the cases when the smallest color class has cardinality two or three respectively. As usual in the area, we will use as tools to solve those later cases some inverse results in Additive Number Theory  presented in Section \ref{sec:ant}.

%%%%%%%%%%%%%%%%%%%%%%%%%%%%%%%%%%%%%%%%%%%%%%%%%%%%%%%%%%%%%%%%%%%%%%%%%%%%%%%%%%%%%%%%%%%%%%%%%%%%%%%%%%%%%%%%%%%%%%%%%%%%%%%%%%%%%%%%%

\section{Notation and preliminaries}\label{sec:pre}

Let $p$ be a prime number and $\Z_p$ be the cyclic group of order $p$. Given a set $S\subseteq \Z_p$ and elements $t,d\in \Z_p$, the sets: $S+t:=\{x+t:x\in S\}$ and  $dS:=\{dx:x\in S\}$ are called the $t$--\emph{translation}  and the $d$--\emph{dilation} of $S$ respectively. Concerning the multiplicative group $\Z_p^{\ast}:=\Z_p\setminus \{0\}$, for every $d\in \Z_p^{\ast}$ we denote by $d^{-1}$ the multiplicative inverse of $d$, and by $\langle d\rangle$ the subgroup of $\Z_p^{\ast}$ generated by $d$.  We say that a subset $S\subseteq \Z_p^{\ast}$ is  \emph{$\langle d\rangle$--periodic} if it is invariant up to $d$--dilation, that is $S=dS$. Note that, a set which is $\langle d\rangle$--periodic  is a union of cosets of $\langle d\rangle$. A set $S$ which is $\langle -1\rangle$--periodic is also called \emph{symmetric}. 

Let $d, t\in \Z_p$, $d\neq 0$, and $S\subseteq \Z_p$, throughout the paper we will work with the transformation $T_{d,t}:\Z_p \to \Z_p$ defined as: 
\begin{center}
$T_{d,t}(S)=dS+t=\{dx+t:x\in S\}$.
\end{center}

Naturally, a set $S\subseteq \Z_p$ is called \emph{invariant} up to $T_{d,t}$, if $T_{d,t}(S)=S$. %We shall note that  a set which is invariant up to $T_{d,t}$ is a suitable translation of a $\langle d\rangle$--periodic set. 
The following observation is not difficult to prove:

\begin{observation}\label{rem:inv_dt} \hspace{.2cm}
\begin{itemize}
\item If $d=1$, $T_{d,t}$ is a $t$--translation.

\item If $d\neq 1$,  the transformation $T_{d,t}$ has a unique fixed point which  is $t(1-d)^{-1}$.

\item For $d\neq 1$, a set $S\subseteq \Z_p$ is invariant up to $T_{d,t}$ if and only if $S+t(d-1)^{-1}$ is a $\langle d\rangle$--periodic set.
\end{itemize}
\end{observation}

We will work with the most general linear equation on three variables written as:
\begin{equation}\label{eq:general}
a_1x+a_2y+a_3z=b
\end{equation}
which has $x,y,z\in \Z_p$ variables, and $a_1,a_2,a_3,b \in \Z_p$ constants, such that $a_1a_2a_3 \neq 0$. Naturally, a set $\{s_1,s_2,s_3\}$ of elements in $\Z_p$ is a solution of  Equation~(\ref{eq:general}), if for some choice of $\{i,j,k\}=\{1,2,3\}$, $a_1s_i+a_2s_j+a_3s_k=b$. The next  observation will be important later.

\begin{observation}\label{re:sss}
A solution $\{s_1,s_2,s_3\}$ of Equation~(\ref{eq:general}) with $s_i=s_j:=s$ for some $i\neq j$, is such that $s_1=s_2=s_3$ if and only if $s(a_1+a_2+a_3)=b$.
\end{observation}

A $3$--\emph{coloring} of $\Z_p$ is a partition $\Z_p=A\cup B\cup C$ with nonempty color classes. A solution of Equation~(\ref{eq:general})  is \emph{rainbow}, if the elements belong pairwise to distinct color classes. A $3$--coloring of $ \Z_p$ is said to be \emph{rainbow--free for Equation~(\ref{eq:general})} if it contains no rainbow solution of Equation~(\ref{eq:general}).

In \cite{jetall} it was proved that every $3$--coloring $\Z_p=A\cup B\cup C$ with $4\leq |A|\leq|B|\leq |C|$ has a rainbow solution of Equation~(\ref{eq:general}) except when $a_1=a_2=a_3$. 

\begin{theorem}[Jungi\'{c} et al. Theorem 6 of \cite{jetall}]\label{thm:m4}
Let $a_1,a_2,a_3,b \in \Z_p$ with $a_1a_2a_3\neq 0$. Then every partition of $Z_p=A\cup B\cup C$ with $|A|,|B|,|C|\geq 4$ contains a rainbow solution of $a_1x+a_2y+a_3z=b$ with the only exception being when $a_1=a_2=a_3$, and every color class is an arithmetic progression with the same common difference $d$, such that $d^{-1}A=\{i\}_{i=t_1}^{t_2-1}$, $d^{-1}B=\{i\}_{i=t_2}^{t_3-1}$, and $d^{-1}C=\{i\}_{i=t_3}^{t_1-1}$, where $(t_1+t_2+t_3)=b+1$, or $b+2$.
\end{theorem}

We shall note that this description has an error. It works well when $d=1$, but do not for other values of $d$. Take for instance $\Z_{13}=A\cup B\cup C$ with $A=\{2,4,6,8\}$, $B=\{10,12,1,3\}$, and $C=\{5,7,9,11,0\}$ (three arithmetic progressions with difference $d=2$), which is a rainbow--free coloring for $x+y+z=2$, and does not satisfy the condition mentioned above. In Theorem~\ref{thm:a1=a2=a3} we corrected the statement.

In \cite{llm1} the case  when some $a_i=a_j$ and $b=0$ in Equation~(\ref{eq:general}) was considered. The authors  provided the description of rainbow--free colorings in this particular case with no restrictions on the size of the color classes.  

\begin{theorem}[Llano, Montejano. Theorem 2 of \cite{llm1}]\label{thm:llm}
A $3$--coloring $\Z_p=A\cup B\cup C$ with $1\leq |A|\leq |B|\leq |C|$ is rainbow--free for $x+y=cz$ if and only if, up to dilation, one of the following holds true:

\begin{itemize}
\item[1.] $A=\{0\}$, with both $B$ and $C$ symmetric $\langle c\rangle$--periodic subsets.

\item[2.] $A=\{1\}$ for

\begin{itemize}
\item[i)] $c=2$, with $(B-1)$ and $(C-1)$  symmetric $\langle 2\rangle$--periodic subsets;

\item[ii)] $c=-1$, with $(B\setminus \{-2\})+2^{-1}$ and $(C\setminus \{-2\})+2^{-1}$ symmetric subsets.
\end{itemize}

\item[3.] $|A|\geq 2$, for $c=-1$, with  $A$, $B$, and $C$  arithmetic progressions with difference $1$, such that $A=\{i\}_{i=t_1}^{t_2-1}$, $B=\{i\}_{i=t_2}^{t_3-1}$, and $C=\{i\}_{i=t_3}^{t_1-1}$, where $(t_1+t_2+t_3)=1$ or $2$.
\end{itemize}
\end{theorem}

We will use both Theorems~\ref{thm:m4} and \ref{thm:llm} in order to fully characterize the structure of rainbow--free colorings concerning Equation~(\ref{eq:general}).

Next we prove that, besides the case when $a_1+a_2+a_3=0$, the structure of rainbow--free colorings of $ \Z_p$ concerning Equation (\ref{eq:general}) with $b\neq 0$ is the same as the structure of rainbow--free colorings for Equation (\ref{eq:general}) with $b= 0$ up to a suitable translation.

\begin{lemma}\label{lem:b}
Let $a:=a_1+a_2+a_3\neq 0$, and let $T:=T_{1,-ba^{-1}}$. Then, $\Z_p=A\cup B\cup C$ is rainbow--free for Equation (1), if and only if $\Z_p=T(A)\cup T(B)\cup T(C)$ is rainbow--free for equation $a_1x+a_2y+a_3z=0$.
\end{lemma}

\begin{proof}
The set $\{s_1,s_2,s_3\}$ is a solution of Equation (1), if and only if the
set $\{T(s_1),T(s_2),T(s_3)\}$ is a solution of $a_1x+a_2y+a_3z=0$.

\end{proof}

Lemma~\ref{lem:b} indicates that, if $a_1+a_2+a_3\neq 0$ then it is sufficient to study rainbow--free colorings of Equation~(\ref{eq:general}) for $b=0$.

%%%%%%%%%%%%%%%%%%%%%%%%%%%%%%%%%%%%%%%%%%%%%%%%%%%%%

\section{Results}\label{sec:result}

First we consider Equation (1) with $a_1=a_2=a_3$. That is, we describe all rainbow--free colorings of $\Z_p$ concerning equation $x+y+z=b$.  This equation is the only one admitting rainbow--free colorings with large color classes. The next theorem is a direct consequence of Theorem~\ref{thm:llm}, and Lemma~\ref{lem:b}.

\begin{theorem}\label{thm:a1=a2=a3}
For $p>3$, a $3$--coloring $\Z_p=A\cup B\cup C$ with $1\leq |A|\leq |B|\leq |C|$ is rainbow--free with respect to equation $x+y+z=b$, if and only if one of the following holds true:

\begin{itemize}
\item[i)] $A=\{s\}$ with both $(B\setminus \{b-2s\})+(s-b)2^{-1}$ and $(C\setminus
\{b-2s\})+(s-b)2^{-1}$ symmetric sets.

\item[ii)] $|A|\geq 2$ and all $A$, $B$ and $C$ are arithmetic
progressions with the same common difference $d$, so that
$d^{-1}A=\{i\}_{i=t_1}^{t_2-1}$, $d^{-1}B=\{i\}_{i=t_2}^{t_3-1}$, and
$d^{-1}C=\{i\}_{i=t_3}^{t_1-1}$ satisfy $(t_1+t_2+t_3)\in
\{1+d^{-1}b,2+d^{-1}b\}$.
\end{itemize}
\end{theorem}

\begin{proof}
For $b=0$ we deduce the structure of rainbow--free colorings from Theorem~\ref{thm:llm}. Since $a_1+a_2+a_3\neq 0$, we use Lemma~\ref{lem:b} to complete the structure characterization.
\end{proof}

Consider now Equation~(\ref{eq:general}) with  some $a_i\neq a_j$. In contrast with Theorem~\ref{thm:a1=a2=a3} we find that all rainbow--free colorings are such that there is one color class of cardinality one. We will let this color class  be $A=\{s\}$. Before stating our main result we define for every $i\in \{1,2,\dots, 6\}$ the transformation $T_i:\Z_p\to \Z_p$ as $T_i(x):=T_{d_i,t_i}(x)$ where:

\begin{center}

$d_1=-a_3a_1^{-1}$ \hspace{1cm} $t_1=(b-a_2s)a_1^{-1}$

$d_2=-a_2a_1^{-1}$ \hspace{1cm} $t_2=(b-a_3s)a_1^{-1}$

$d_3=-a_1a_2^{-1}$ \hspace{1cm} $t_3=(b-a_3s)a_2^{-1}$

$d_4=-a_3a_2^{-1}$ \hspace{1cm} $t_4=(b-a_1s)a_2^{-1}$

$d_5=-a_1a_3^{-1}$ \hspace{1cm} $t_5=(b-a_2s)a_3^{-1}$

$d_6=-a_2a_3^{-1}$ \hspace{1cm} $t_6=(b-a_1s)a_3^{-1}$

\end{center}

\begin{theorem}\label{thm:main}
A $3$--coloring $\Z_p=A\cup B\cup C$ with $1\leq |A|\leq |B|\leq |C|$ is rainbow--free for equation:
\begin{equation}\label{eq:general2}
a_1x+a_2y+a_3z=b, \hspace{.2cm}with \hspace{.2cm}some \hspace{.2cm} a_i\neq a_j,
\end{equation}
if and only if $A=\{s\}$ with $s(a_1+a_2+a_3)=b$, and both $B$ and $C$ are sets invariant up to $T_i$ for every $i\in\{1,2,\dots ,6\}$.
\end{theorem}

\begin{proof}
The proof is deduced from Theorem~\ref{thm:m4}, and the lemmas in Sections~\ref{sec:A=2} and \ref{sec:A=3}.
Let $\Z_p=A\cup B\cup C$ with $1\leq |A|\leq |B|\leq |C|$ be a rainbow--free coloring of Equation (\ref{eq:general2}). Then Theorem~\ref{thm:m4} implies that $|A|\in\{1,2,3\}$.  From Lemmas~\ref{lem:A2B3} and~\ref{lem:A2B2}  (concerning the case $|A|=2$), and Lemmas~\ref{lem:A3B4} and~\ref{lem:A3B3} (concerning the case $|A|=3$) we deduce that actually $|A|=1$.  The rest  of the proof follows by Lemma~\ref{lem:A=1} (concerning the case $|A|=1$). 
\end{proof}

Given a prime number $p$, an equation will be called  \emph{rainbow} with respect to $p$, if every $3$--coloring of $\Z_p$  contains a rainbow solution of it. Consequently, a  \emph{non--rainbow} equation with respect to a  prime number $p$ is an equation such that there are rainbow--free colorings of $\Z_p$ with respect to the same. For instance,  $x+y+z=b$ is a non--rainbow equation with respect to all primes, and $x+y=2z$ is a rainbow equation with respect to $p$, if and only if $p$ satisfies either $|\langle 2\rangle|=p-1$, or $|\langle 2\rangle|=(p-1)/2$ where $(p-1)/2$ is an odd number (see Theorem $5$ of \cite{jetall}, or Corollary 1 of \cite{llm1}).

We can deduce  from Theorem~\ref{thm:main}  which equations are rainbow (hence,  which ones are non--rainbow). We generalize the above result about $3$--term arithmetic progressions in the next corollary. Before continuing we  highlight an important  consequence of Observation~\ref{rem:inv_dt}.

\begin{observation}\label{rem:final}
If $s(a_1+a_2+a_3)=b$ then,  for every $i\in\{1, ... ,6\}$, the fixed point of  $T_i$ is precisely $A=\{s\}$. Moreover, a set which is invariant up to $T_i$ for every $i\in\{1,2,\dots ,6\}$ is a suitable translation of a $\langle d_1, d_2,...,d_6\rangle$--periodic set.
\end{observation}

Theorem~\ref{thm:main} can  also be stated in the opposite manner, such as:  A coloring $Z_p=A\cup B\cup C$ has a rainbow solution of Equation (\ref{eq:general2}), if and only if some of the following holds true:

\begin{itemize}
\item[i)] $2\leq \min \{|A|,|B|,|C|\}$.

\item[ii)] $a_i+a_2+a_3=0\neq b$.

\item[iii)] $T_i(X)\neq X$ for some $i\in\{1,...,6\}$, and $X\in\{B,C\}$.
\end{itemize}

\begin{corollary}\label{thm:main2}
Every $3$--coloring of $\Z_p$ with nonempty color classes contains a rainbow solution of Equation~(\ref{eq:general2}), if and only if one of the following holds true:

\begin{itemize}
\item[i)] $a_1+a_2+a_3=0\neq b$  
\item[ii)] $|\langle d_1, d_2,...,d_6\rangle|=p-1$.
\end{itemize}
\end{corollary}

\begin{proof}
If $a_1+a_2+a_3=0\neq b$  then the second point in the previous paragraph  implies that every $3$--coloring of $\Z_p$ contains a rainbow solution of Equation (\ref{eq:general2}). If $|\langle d_1, d_2,...,d_6\rangle|=p-1$ then by Observation~\ref{rem:final} it is impossible to simultaneously satisfy $T_i(B)=B$, and $T_i(C)=C$ for every $i\in\{1,...,6\}$. Thus, by the third point in the previous paragraph we conclude the desired implication.

On the other hand, suppose that $a_1+a_2+a_3=0= b$, or $a_1+a_2+a_3\neq0$, and  $|\langle d_1, d_2,...,d_6\rangle|<p-1$. Then, according to Theorem~\ref{thm:main} there exist rainbow-free colorings  of $\Z_p$ with nonempty color classes.
\end{proof}

Corollary \ref{thm:main2} gives a criteria to determine which equations are rainbow. To finish this section we describe with two particular examples how to construct rainbow--free colorings of $\Z_p$ for equations provided with the following  conditions: $a_1+a_2+a_3=0= b$ or $a_1+a_2+a_3\neq0$, and  $|\langle d_1, d_2,...,d_6\rangle|<p-1$.

\begin{example}
Consider $\Z_{13}$ and the equation $x-4y+3z=0$. Since $1-4+3=0=b$, in order to construct a rainbow--free coloring,  we can let $A$ be any point of $\Z_{13}$. Let $A=\{0\}$. Since $\langle d_1, d_2,...,d_6\rangle =\{\pm1,3,4\}$ (so, $|\langle d_1, d_2,...,d_6\rangle|\neq 12$),  we can partition $\Z_{13}\setminus \{0\}=B\cup C$ in such a way that both $B$ and $C$ are  $\langle d_1, d_2,...,d_6\rangle$--periodic sets. We  let $B=\{1,3,4,9,10,12\}$ and $C=\{2,5,6,7,8,11\}$. In this case, also any translation of such coloring will be rainbow--free.

\end{example}

\begin{example}
Consider $\Z_{17}$ and the  equation $x+8y-2z=3$. Since $1+8-2= 7\neq 0$ then, in order to construct a rainbow--free coloring, we let $A=\{(3) (7^{-1})\}=\{15\}$. Since $\langle d_1, d_2,...,d_6\rangle=\{\pm1,2\}$ (so, $|\langle d_1, d_2,...,d_6\rangle|\neq 16$) we can partition $\Z_{17}\setminus \{15\}=B\cup C$ in such a way that both $B$ and $C$ are translations of  $\langle d_1, d_2,...,d_6\rangle$--periodic sets. We  let $B=\{16,0,2,6,7,11,13,14\}$ and $C=\{1,3,4,5,8,9,10,12\}$.
\end{example}

%%%%%%%%%%%%%%%%%%%%%%%%%%%%%%%%%%%%%%%%%%%%%%%%%%%%%

\section{The case $|A|=1$}\label{sec:A=1}

In this section we describe all rainbow--free colorings of $\Z_p$ concerning Equation~(\ref{eq:general2}) with a color class of cardinality one. Throughout the section, $\Z_p=A\cup B\cup C$ will be a $3$--coloring of $\Z_p$ with $A=\{s\}$.

\begin{lemma}\label{lem:basic}
Let  $\Z_p=\{s\}\cup B\cup C$ be a rainbow--free coloring for Equation~(\ref{eq:general2}). Then

$$B\subseteq T_i(B)\cup \{d_is+t_i\}$$
$$C\subseteq T_i(C)\cup \{d_is+t_i\}$$
for every $i\in\{1,2,\dots ,6\}$.

\end{lemma}

\begin{proof}
Consider the partition $\Z_p=T_i(s)\cup T_i(B)\cup T_i(C)$, and suppose $B\nsubseteq T_i(B)\cup T_i(s)$ then $B\cap T_i(C)\neq \emptyset$. Thus, there exist $u\in B$ and $v\in C$ such that $u=d_iv+t_i$. It is not hard to see then, that $\{u,v,s\}$ will be a rainbow solution of Equation~(\ref{eq:general2}).
\end{proof}

The aim of this section is to prove that, in fact, a rainbow--free coloring with a color class of cardinality one, is such that the remained color classes are invariant up to $T_i$. It is not difficult to see that $\Z_p=\{s\}\cup B\cup C$ with $T_i(B)=B$ and $T_i(C)=C$ is s rainbow--free coloring. We will prove this and the converse. The converse provided with a restriction on $s$ in terms of $a_1,a_2,a_3$ and $b$.

\begin{lemma}\label{lem:A=1}
A $3$--coloring $\Z_p=\{s\}\cup B\cup C$ is rainbow--free for Equation~(\ref{eq:general2}) if and only if $s$ is such that
\begin{equation}\label{eq:s}
s(a_1+a_2+a_3)=b
\end{equation}

and, $$T_i(B)=B$$
\begin{equation}\label{eq:Ti}
T_i(C)=C
\end{equation}
for every $i\in\{1,2,\dots ,6\}$.
\end{lemma}

\begin{proof}
Assume without lost of generality that $a_2\neq a_3$. The ``if'' part of the statement follows since, for both $X\in\{B,C\}$, every solution $\{s_1,s_2,s_3\}$ of Equation~(\ref{eq:general2}) that has one element in $A$ and other element in $X$ is such that $\{s_1,s_2,s_3\}\subseteq A\cup X$ so the $3$--coloring is rainbow free.

Conversely, assume that the $3$--coloring is rainbow--free. Suppose first that $d_is+t_i=s$ for some $i\in \{1,2,\dots, 6\}$.
Then, it is not hard to see that (\ref{eq:s}) is satisfied. Since
the coloring is rainbow--free, any solution $\{s,u,v\}$ with $u\in
B$ is such that $v\in B\cup A$, but Observation~\ref{re:sss} together
with (\ref{eq:s}) indicate that, actually $v\in B$. Therefore
$T_i(B)\subseteq B$ for every $i\in \{1,2,\dots, 6\}$. The same is
true for $C$, and by cardinality we get (\ref{eq:Ti}).

Assume now that $d_is+t_i\neq s$ for every $i\in \{1,2,\dots, 6\}$.
With out loss of generality let $d_1s+t_1\in C$. Since $B\subseteq
T_1(B)\cup \{d_1s+t_1\}$ (Lemma~\ref{lem:basic}) then $B\subseteq
T_1(B)$. Thus $T_1(B)=B$, that is
\begin{equation}\label{eq:B1}
B=d_1B+t_1.
\end{equation}

Note that $d_1s+t_1=-a_3a_1^{-1}s+ba_1^{-1}-a_2a_1^{-1}s=d_2s+t_2$
then, by the same arguments, we get $T_2(B)=B$, that is
\begin{equation}\label{eq:B2}
B=d_2B+t_2.
\end{equation}

By a dilation of (\ref{eq:B2}) we get

\begin{equation}\label{eq:B3}
d_1B=d_2d_1B+d_1t_2.
\end{equation}

By a translation of (\ref{eq:B2}) we get $B-t_1=d_2B+t_2-t_1$ which
can be expressed as

\begin{equation}\label{eq:B4}
B-t_1=d_2(B-t_1)+(t_2-t_1+d_2t_1).
\end{equation}

By (\ref{eq:B1}) we know that $B-t_1=d_1B$, then it follows from
(\ref{eq:B3}) and (\ref{eq:B4}) that necessarily

$$d_1t_2=t_2-t_1+d_2t_1,$$

from which, by simple calculation, we get (\ref{eq:s}), a
contradiction by assumption.
\end{proof}

%%%%%%%%%%%%%%%%%%%%%%%%%%%%%%%%%%%%%%%%%%%%%%%%%%%%%

\section{Additive tools}\label{sec:ant}

Before treating the remaining cases $|A|=2$ and $|A|=3$, we give some results in Additive Number Theory. These results have been used previously  \cite{jetall,llm1} in solving arithmetic anti-Ramsey problems. The following is the main idea to do it:

As usual, for sets $X,Y\subseteq \Z_p$, let $X+Y=\{x+y:x\in X, y\in Y\}$. The well known Cauchy--Davenport's Theorem \cite{nat} states that for any $X,Y\subseteq \Z_p$ with $X+Y\neq \Z_p$, it happens that $|X+Y|\geq |X|+|Y|-1$. On the other hand, $\Z_p=A\cup B \cup C$  is a rainbow--free coloring, if and only if $a_iX+a_jY\cap -a_kZ+b=\emptyset$ for every $\{i,j,k\}=\{1,2,3\}$ and $\{X,Y,Z\}=\{A,B,C\}$. Equally:

\begin{equation}\label{eq:cont}
a_iX+a_jY\subseteq \Z_p\setminus (-a_kZ+b).
\end{equation}

Since $p=|X|+|Y|+|Z|$, and $|-a_kZ+b|=|Z|$, from (\ref{eq:cont}) on one side, and Cauchy--Davenport on the other we obtain:

\begin{equation}\label{eq:basic}
|X|+|Y|-1\leq |a_iX+a_jY|\leq |X|+|Y|
\end{equation}

The next two important results characterize the structure of subsets in $\Z_p$ with $|X+Y|=|X|+|Y|-1$, and $|X+Y|=|X|+|Y|$ respectively.

\begin{theorem}[Vosper \cite{vos}]\label{thm:vos}
Let $X,Y\subseteq \Z_p$ with $|X|,|Y|\geq 2$, and $$|X+Y|=|X|+|Y|-1\leq p-2.$$ Then both $X$ and $Y$ are arithmetic progressions with the same common difference.
\end{theorem}

An \emph{almost arithmetic progression} with \emph{difference} $d$ in $\Z_p$ is an arithmetic progression with difference $d$, and one term removed. Observe that an arithmetic progression is an almost arithmetic progression, if the term removed is the initial or the final term of the original progression.

\begin{theorem}[Hamidoune--R{\o}dseth \cite{hr}]\label{thm:hr}
Let $X,Y\subseteq \Z_p$ with $|X|,|Y|\geq 3$ and $$7\leq |X+Y|=|X|+|Y|\leq p-4.$$ Then both $X$ and $Y$ are almost arithmetic progressions with the same common difference.
\end{theorem}

We will also need the following technical lemma. We called an arithmetic progression with difference $d=1$, an \emph{interval}. 

\begin{lemma}\label{lem:technical}
Let $X\subseteq \Z_p$ with $5 \leq |X| \leq p-5$. If both $X$ and $tX$ are the union of at most two arithmetic progressions with the same common difference $d$, then $t\in \{0, \pm 1, \pm2, \pm 2^{-1}\}$. 
\end{lemma}

\begin{proof}
We may assume, with out loss of generality, that $5 \leq |X| \leq \frac{p-1}{2}$, otherwise we take $ \Z_p\setminus X$. Also we suppose that $d=1$, that is,  $X$ is  the union of at most two intervals, otherwise we analyze $d^{-1}X$. By hypothesis, $Y=tX$ is also the union of at most two intervals, $Y_1$ and $Y_2$. Note that $|(X+1)\setminus X|\leq 2$, and so:  
\begin{equation}\label{eq:lem}
|(Y+t)\setminus Y|\leq 2.
\end{equation}
From here we consider two cases. Suppose first that either $Y_1\cap (Y_1+t)\neq \emptyset$ or $Y_2\cap (Y_2+t)\neq \emptyset$. Assume without loss of generality $Y_1\cap (Y_1+t)\neq \emptyset$. Then, since $|Y | \leq \frac{p-1}{2}$, it follows that $|t|\leq |(Y+t)\setminus Y|$, which together with (\ref{eq:lem}) implies  that $t\in \{0, \pm 1, \pm2\}$. Suppose now that $Y_1\cap (Y_1+t) =\emptyset$ and $Y_2\cap (Y_2+t) =\emptyset$. Then, by  (\ref{eq:lem})  it follows that:
\begin{equation}\label{eq:lem2}
|(Y_1+t)\setminus Y_2|+|(Y_2+t)\setminus Y_1|\leq 2 
\end{equation}
We shall note that (\ref{eq:lem2}) implies that the cardinalities of $Y_1$ and $Y_2$ differ at most in $2$ elements. Moreover, letting $Y_1=[y_1,y_2]$, and $Y_2=[y_3,y_4]$, it must be that: $y_1+t=y_3+\epsilon_1$, with $|\epsilon_1|\leq 2$; and $y_3+t=y_1+\epsilon_2$, with $|\epsilon_2|\leq 2$; where $|\epsilon_1|+|\epsilon_2|\leq 2$. From which it follows that  $2t=\epsilon_1+\epsilon_2$, and according to $|\epsilon_1|+|\epsilon_2|\leq 2$,  in all possible cases we get $t\in \{0, \pm 1, \pm2, \pm 2^{-1}\}$.

\end{proof}

%%%%%%%%%%%%%%%%%%%%%%%%%%%%%%%%%%%%%%

\section{The case $|A|=2$}\label{sec:A=2}

In this section we prove that there are no rainbow-free colorings of $\Z_p$ concerning Equation~(\ref{eq:general2}) such that the smallest  color class has two elements. First let us note a useful fact. 

\begin{lemma}\label{lem:useful}
Let $\Z_p=A\cup B \cup C$  be a rainbow--free coloring with $|A|=2\leq |B|\leq |C|$. 
For any choice of  $\{i,j,k\}=\{1,2,3\}$, the sets  $a_iB$, $a_iC$, $a_jB$ and $a_jC$ are unions of at most two arithmetic progressions with difference $d$, where $d$ is the difference between the two elements in $a_kA$.
\end{lemma}

\begin{proof}
It follows from (\ref{eq:basic}) that, for any choice of $\{i,j,k\}=\{1,2,3\}$, we have: 

\begin{equation}\label{eq:2b}
|B|+1\leq |a_iA+a_jB| \leq |B|+2. 
\end{equation}

Suppose that the difference between the  two elements in $a_iA$ is $d$. Since $|a_jB|=|B|$, then necessarily $a_jB$ is the union of at most two arithmetic progressions with difference $d$, and the same will be true for $a_kB$. By repeating this argument we conclude the claim.
\end{proof}

Next, we consider the case where two of the coefficients in Equation~(\ref{eq:general2}) are equal. That is, with out loss of generality, we handle equation: $x+y+cz=b$ where $c\neq 1$.   

\begin{proposition}\label{prop:1}
Every $3$--coloring $\Z_p=A\cup B \cup C$ with $|A|=2$ contains a rainbow solution of $x+y+cz=b$ where $c\neq 1$.
\end{proposition}

\begin{proof}
By Lemma~\ref{lem:b} and Theorem~\ref{thm:llm} the statement is true for $c\neq -2$. If $c=-2$, then Lemma~\ref{lem:useful} states that $B$, $C$, $-2B$ and $-2C$ are all sets which are union of at most two arithmetic progressions with difference $d$, where $d$ is the difference between the two elements of $A$. By the sake of comprehension we will assume that $d=1$, otherwise we can analyze the form of the partition $\Z_p=d^{-1}A\cup d^{-1}B\cup d^{-1}C$. Recall we called an arithmetic progression with difference one, an \emph{interval}. Let $A=\{t,t+1\}.$

\textbf{Case 1.}   Some  $X\in \{B,C\}$, say $B$, is an interval. Since $2\leq |B|\leq p-4$ and $-2B$ is union of at most two intervals then necessarily  $|B|=2$. Moreover, since also $-2C$ is union of at most two intervals, then actually $B=\{t+2^{-1},t+2^{-1}+ 1\}$ or  $B=\{t+2^{-1},t+2^{-1}- 1\}$. Note that in both cases $-2B$ is a two element set whose difference is $2$. Recall now that a rainbow--free coloring for $x+y-2z=b$ satisfies $2B+b\subseteq \Z_p\setminus (A+C)$, which is a contradiction, since  $\Z_p\setminus (A+C)$ is a two element set whose difference is $2^{-1}$, and  $2B+b$ (as well as $-2B$) is a two element set whose difference is $2$. 

\textbf{Case 2.}   Both $B$ and $C$ are not intervals, they are  union of exactly two intervals. Suppose with out loss of generality that $t+2^{-1}\in B$.  Then, it is not hard to see that  either $B=\{t+2^{-1}+i\}_{i=0}^{i=k}\cup \{t+2+i\}_{i=0}^{i=k-1}$, $B=\{t+2^{-1}+i\}_{i=0}^{i=k+1}\cup \{t+2+i\}_{i=0}^{i=k-1}$, $B=\{t+2^{-1}-i\}_{i=0}^{i=k}\cup \{t-1-i\}_{i=0}^{i=k-1}$ or $B=\{t+2^{-1}-i\}_{i=0}^{i=k+1}\cup \{t-1-i\}_{i=0}^{i=k-1}$,  where $1\leq k\leq \frac{p+1}{2}-3$. In any case $-2C$ is an interval. Recall now that a rainbow--free coloring for $x+y-2z=b$ satisfy $2C+b\subseteq \Z_p\setminus (A+B)$, which is a contradiction.
 \end{proof}

Next, we consider two more specific equations that arise naturally from the proofs of  Lemmas~\ref{lem:A2B3} and~\ref{lem:A2B2} below. 

\begin{proposition}\label{prop:2}
Every $3$--coloring $\Z_p=A\cup B \cup C$ with $|A|=2$ contains a rainbow solution of $x-y+2z=b$ (respectively, $x-y+2^{-1}z=b$).
\end{proposition}

\begin{proof}
The proof is analogous to the proof of the previous proposition. Concerning the equation $x-y+2z=b$ (respectively, $x-y+2^{-1}z=b$) by Lemma~\ref{lem:useful} we know $B$, $C$, $2B$ and $2C$  (respectively,  $B$, $C$, $2^{-1}B$ and $2C^{-1}$ ) are union of at most two arithmetic progressions with difference $d$, where $d$ is the difference between the two elements of $-A$.
\end{proof}

Now, we are ready to prove the lemmas who dismiss in general  the existence of rainbow--free colorings with the smallest color class of size two.

\begin{lemma}\label{lem:A2B3}
Every $3$--coloring $\Z_p=A\cup B \cup C$ with $|A|=2$ and $3\leq |B|\leq |C|$ contains a rainbow solution of Equation~(\ref{eq:general2}).
\end{lemma}

\begin{proof} 
Suppose for a contradiction that  $\Z_p=A\cup B \cup C$ is a rainbow--free coloring for Equation~(\ref{eq:general2}) with $|A|=2$ and $3\leq |B|\leq |C|$. Then $5\leq |C| \leq p-5$ and, by Lemma~\ref{lem:useful}  both $a_1C$ and $a_2C$ are union of at most two arithmetic progressions with the same common difference. From   Lemma~\ref{lem:technical},  since $a_1C=a_1a_2^{-1}(a_2C)$, we conclude that $a_1a_2^{-1}\in \{\pm 1,\pm 2, \pm 2^{-1} \}$. With  similar arguments we obtain that:

\begin{equation}\label{eq:restrictive}
\{a_1a_2^{-1},a_2a_3^{-1}, a_3a_1^{-1}\}\subseteq \{\pm 1,\pm 2, \pm 2^{-1} \}. 
\end{equation}

If  $a_i=a_j$  for some distinct $i,j\in\{1,2,3\}$, we get a contradiction by Proposition~\ref{prop:1}. Assume then, with out lost of generality,  that $a_1=1$ and all three coefficients are different from each other.  Hence, by (\ref{eq:restrictive}) we have $a_2,a_3\in\{-1,\pm 2, \pm 2^{-1}\}$. If $a_2=-1$ then $a_3\in\{\pm 2, \pm 2^{-1}\}$. Note that $a_3=2$ gives an equivalent equation than $a_3=-2$, and the same is true for $a_3=2^{-1}$ or $a_3=-2^{-1}$, thus we obtain either $x-y+2z=b$  or $x-y+2^{-1}z=b$; in both cases we obtain a contradiction by Proposition~\ref{prop:2}. The remaining cases where $a_1=1$ and $a_2,a_3\in\{\pm 2,\pm 2^{-1}\}$ all give an equation equivalent to one of the considered in  Propositions~\ref{prop:1} and~\ref{prop:2}.
\end{proof}

\begin{lemma}\label{lem:A2B2}
Every $3$--coloring $\Z_p=A\cup B \cup C$ with $|A|=|B|=2$ and $|B|\leq |C|$ contains a rainbow solution of Equation~(\ref{eq:general2}). 
\end{lemma}

\begin{proof} 
Suppose for a contradiction that  $\Z_p=A\cup B \cup C$ is a rainbow--free coloring for Equation~(\ref{eq:general2}), with $|A|=|B|=2$. By (\ref{eq:2b}) we get $3\leq |a_iA+a_jB| \leq 4$. 

Assume first that for some pair of coefficients, say $a_1$ and $a_2$, we have  $|a_1A+a_2B|=3$. Then Vosper's Theorem (Theorem~\ref{thm:vos}) establishes that the sets $a_1A$ and $a_2B$ are arithmetic progressions with same common difference. Let $a_1A=\{t_1,t_1+d\}$ and $a_2B=\{t_2,t_2+d\}$. Consider now the set $a_2A+a_1B$, which can be written as:

\begin{equation}\label{eq:2+2}
\{a_2a_1^{-1}t_1,a_2a_1^{-1}(t_1+d)\}+\{a_1a_2^{-1}t_2,a_1a_2^{-1}(t_2+d)\}
\end{equation}
since $a_2A=a_2a_1^{-1}(a_1A)$ and $a_1B=a_1a_2^{-1}(a_2B)$. By  (\ref{eq:cont}) both $a_1A+a_2B$ and $a_2A+a_1B$ are contained in $ \Z_p\setminus (-a_3C+b)$ which is a four elements set. If $|a_2A+a_1B|=3$ then Vosper Theorem establishes $a_2a_1^{-1}\in\{\pm a_1a_2^{-1}\}$ and it follows plainly that $a_2a_1^{-1}=1$ providing a contradiction by Proposition\ref{prop:1}. Assume $|a_2A+a_1B|=4$.
Then $a_1A+a_2B$ is contained in $a_2A+a_1B$. Since $a_1A+a_2B$ is a three--term arithmetic progression with difference $d$, by  analyzing the set of differences in (\ref{eq:2+2}) we get that either $a_2a_1^{-1}d=d$ or $a_1a_2^{-1}d=d$. In both cases  $a_1=a_2$, which is a contradiction by Proposition~\ref{prop:1}.

Assume now that $|a_iA+a_jB|=4$ for all distinct $i,j\in\{1,2,3\}$. Let $a_1A=\{t_1,t_1+d_1\}$ and $a_2B=\{t_2,t_2+d_2\}$ with $d_1\neq \pm d_2$. As in the previous paragraph, we note that $a_2A+a_1B$ can be written as: 

\begin{equation}\label{eq:2+2,2}
\{a_2a_1^{-1}t_1,a_2a_1^{-1}(t_1+d_1)\}+\{a_1a_2^{-1}t_2,a_1a_2^{-1}(t_2+d_2)\}
\end{equation}

Again, by comparing the set of differences in (\ref{eq:2+2,2}) with the set of differences in $a_1A+a_2B$ we deduce that, either $d_1=a_2a_1^{-1}d_1$, or $d_1=a_1a_2^{-1}d_2$. In the first case we get a contradiction by  Proposition~\ref{prop:1}. In the second case, %by substituting $d_1$ in $a_1A=\{t_1,t_1+d_1\}$ and compute, we get  that $A=\{a_1^{-1}t_1,a_1^{-1}t_1+a_2^{-1}d_2\}$ and $B=\{a_2^{-1}t_2, a_2^{-1}t_2+a_2^{-1}d_2 \}$, which is a contradiction since we assume that the difference between the elements in $A$ was distinct to the difference  between the elements in $B$. 
$(a_1A+a_2B\cup a_2A+a_1B)\subseteq \Z_p\setminus (-a_3C+b)$ implies $a_1A+a_2B=a_2A+a_1B$ hence $t_1+t_2=a_2a_1^{-1}t_1+a_1a_2^{-1}t_2$ thus $A=B$ which is impossible. 
\end{proof}

%%%%%%%%%%%%%%%%%%%%%%%%%%%%%%%%%%%%%%%%%%%%%%%%%%%%%%%%%%%%%%%%

\section{The case $|A|=3$}\label{sec:A=3}

In this section we prove that there are no rainbow-free colorings of $\Z_p$ concerning Equation~(\ref{eq:general2}) such that the smallest  color class has three elements. In the case $|A|=3$ and $4\leq |B|\leq |C|$, we will follow a similar line of argument than in the previous section. For the case $|A|=|B|=3$, we use some other technical lemmas. First let us note a useful fact. 

\begin{observation}\label{rem:useful2}
If $p\geq 11$, $|X|=3$ and $X$ is an almost arithmetic progression of difference $d$, then $X$ is an almost arithmetic progression of difference $d'$ if and only if $d\in\{\pm d'\}$
\end{observation}

\begin{lemma}\label{lem:useful3}
Let $\Z_p=A\cup B \cup C$  be a rainbow--free coloring with $|A|=3$ and $4\leq |B|\leq |C|$.
For any choice of $i, j \in\{1,2,3\}$, $i\neq j$, the sets  $a_iB,a_iC,a_jB,a_jC$ are almost arithmetic progression with the same common difference.
\end{lemma}

\begin{proof}
 By (\ref{eq:basic}) we know that for any choice of $\{i,j,k\}=\{1,2,3\}$, and $\{X,Y,Z\}=\{A,B,C\}$,  either $|a_iX+a_jY|=|a_iX|+|a_jY|-1$, or $|a_iX+a_jY|=|a_iX|+|a_jY|$. In the first case it follows from Vosper's Theorem  that both $a_iX$ and $a_jY$ are arithmetic progressions with the same common difference; in the second case, Theorem~\ref{thm:hr} implies that both $a_iX$ and $a_jY$ are almost arithmetic progressions with the same common difference. In both cases we obtain that  the sets $a_iA$ and $a_jX$, with $X\in\{B,C\}$, are almost arithmetic progressions with the same common difference. 
By repeating this argument and the use of the previous observation we conclude the claim.
\end{proof}

As in the previous section we first handle specific cases  that will arise from the lemmas below. 

\begin{proposition}\label{prop:1,3}
Every $3$--coloring $\Z_p=A\cup B \cup C$ with $|A|=3$ and $4\leq|B|\leq|C|$ contains a rainbow solution of $x+y+cz=b$ where $c\neq 1$.
\end{proposition}

\begin{proof}
By Lemma~\ref{lem:b} and Theorem~\ref{thm:llm} the statement is true for $c\neq -2$. If $c=-2$ and we assume there are not rainbow solutions, then Lemma~\ref{lem:useful3} states that $B$, and $-2B$ are almost arithmetic progression with difference $d$. For the sake of simplicity assume $d=1$. Hence $B=\{t+i\}_{i=1}^{i=j-1}\cup \{t+i\}_{i=j+1}^{i=k}$ for some $t\in\Z_p$, $4\leq k\leq p-6$ and $1<j\leq k$ and thereby $-2B=\{-2t-2i\}_{i=1}^{i=j-1}\cup \{-2t-2i\}_{i=j+1}^{i=k}$ which clearly is not an almost arithmetic progression of difference $1$ contradicting the above assumption.
\end{proof}

\begin{proposition}\label{prop:1,3bb}
Every $3$--coloring $\Z_p=A\cup B \cup C$ with $|A|=3$ and $4\leq|B|\leq|C|$ contains a rainbow solution of $x-y+2z=b$ (respectively, $x-y+2^{-1}z=b$).
\end{proposition}

\begin{proof}
The proof is analogous to the proof of the previous proposition. Concerning the equation $x-y+2z=b$ (respectively, $x-y+2^{-1}z=b$) by Lemma~\ref{lem:useful3} we know $B$ and  $2B$   (respectively,  $B$ and $2^{-1}B$) are almost arithmetic progressions with the same common difference.
\end{proof}
 
Now we are ready to prove the lemma who dismiss  the existence of rainbow--free colorings in the case $|A|=3<|B|$.

 \begin{lemma}\label{lem:A3B4}
Every $3$--coloring $\Z_p=A\cup B \cup C$ with $|A|=3$ and $4\leq |B|\leq |C|$ contains a rainbow solution of Equation~(\ref{eq:general2}).
\end{lemma}

\begin{proof}
Suppose for a contradiction that  $\Z_p=A\cup B \cup C$ is a rainbow--free coloring for Equation~(\ref{eq:general2}) with $|A|=3$ and $4\leq |B|\leq |C|$. By Lemma~\ref{lem:useful3} we know that $a_2B$, $a_2C$, $a_3B$ and $a_3C$ are  almost arithmetic progressions with the same common difference. Hence, from   Lemma~\ref{lem:technical} we conclude that $a_2a_3^{-1}\in \{\pm 1,\pm 2, \pm 2^{-1} \}$; in the same way $a_1a_2^{-1}, a_3a_1^{-1}\in \{\pm 1,\pm 2, \pm 2^{-1} \}$. Note that these conditions give precisely the cases considered in Propositions~\ref{prop:1,3} and \ref{prop:1,3bb}: If  $a_i=a_j$  for some distinct $i,j\in\{1,2,3\}$, we get a contradiction by Proposition~\ref{prop:1,3}. Assume then, with out lost of generality,  that $a_1=1$ and all three coefficients are different from each other, thus $a_2,a_3\in\{-1,\pm 2, \pm 2^{-1}\}$. If $a_2=-1$ then $a_3\in\{\pm 2, \pm 2^{-1}\}$. Note that $a_3=2$ gives an equivalent equation than $a_3=-2$, and the same is true for $a_3=2^{-1}$ or $a_3=-2^{-1}$, 
in both cases we obtain a contradiction by Proposition~\ref{prop:1,3}. The remaining cases where $a_1=1$ and $a_2,a_3\in\{\pm 2,\pm 2^{-1}\}$ all give an equation equivalent to one of the considered in  Propositions~\ref{prop:1,3} and~\ref{prop:1,3bb}.

\end{proof}
 
Next we prove a technical lemmas to conclude the remaining case $|A|=|B|=3$. 
 
\begin{lemma}\label{lem:X3Y3}
 Suppose $p\geq 11$ and $X,Y\subseteq \Z_p$. If $|X|=|Y|=3$ and $|X+Y|\in\{5,6\}$, then one of the following holds true
 \begin{itemize}
\item[i)] $X=Y+u$ for some $u\in\Z_p$.
\item[ii)] $\{X,Y\}=\{\{w,w+d,w+2d\},\{u,u+d,u+3d\}\}$ for some $w,u,d\in\Z_p$.
\end{itemize}
\end{lemma}

\begin{proof} 
If $|X+Y|=5$ then Theorem \ref{thm:vos} implies that both $X$ and $Y$ are arithmetic progressions with the same common difference, and therefore $X=Y+u$ for some $u\in\Z_p$. In the other case, $|X+Y|=6$, is tedious but not difficult to prove the claim (for more details see \cite{llm1}).

\end{proof} 

We will need the analogous of Proposition~\ref{prop:1,3} in the more specific case $|A|=|B|=3$.
 
\begin{proposition}\label{prop:1,3b}
Every $3$--coloring $\Z_p=A\cup B \cup C$ with $|A|=|B|=3$ and $|B|\leq|C|$ contains a rainbow solution of $x+y+cz=b$ where $c\neq 1$.
\end{proposition}

\begin{proof}
By Lemma~\ref{lem:b} and Theorem~\ref{thm:llm} the statement is true for $c\neq -2$. So we consider the equation $x+y-2z=b$, and suppose, by contradiction, that $\Z_p=A\cup B \cup C$ with $|A|=|B|=3$ and $|B|\leq|C|$ is a rainbow--free coloring for it.  We handle two cases.

\textbf{Case 1.} There is no element $w\in \Z_p$ such that, either $A=B+w$, or  $A=-2B+w$. Then by Lemma~\ref{lem:X3Y3} we know that $\{A,B\}=\{\{w,w+d,w+2d\},\{u,u+d,u+3d\}\}$. In both cases it is not difficult to see that $\max \{|A-2B|,|B-2A|\}>6$ which is a contradiction by (\ref{eq:basic}).

\textbf{Case 2.} There are $u_1$ and $u_2\in \Z_p$ such that $A=B+u_1$ and $A=-2B+u_2$. Then $B=-2B+u_2-u_1$. By the third outcome of Observation~\ref{rem:inv_dt} we know that there exist a $w\in \Z_p$ such that $B+w$ is invariant up to a $\langle -2\rangle$--dilation. Hence, $B+w=\{x,-2x,4x\}$ for some $x\in\Z_p$, and thereby $(-2)^3x=x$ which is a contradiction.

\end{proof}
 
 \begin{lemma}
\label{lem:lambda}
Let $\lambda\in\Z_p$ be such that $\lambda^4+\lambda^2+1=0$ and 
\begin{equation*}
X:=\{0,1,2,\lambda^2+1,\lambda^2+2,2\lambda^2+2\}.
\end{equation*}
 If $\{w,w+1,w+2\}\subseteq \lambda X$, then either $p\leq 7$ or
 \begin{equation*}
w= \left\{ \begin{array}{lll}
2 \lambda & \mbox{if $\lambda^3=1$} \\
2\lambda-2 & \mbox{if $\lambda^3=-1$}.\end{array} \right.
 \end{equation*}
  \end{lemma}
\begin{proof}
Write $Y:=\{w,w+1,w+2\}$. Since 
\begin{equation*}
(\lambda^2-1)(1+\lambda^2+\lambda^4)=\lambda^6-1=(\lambda^3-1)(\lambda^3+1),
\end{equation*}
we have $\lambda^3\in\{\pm 1\}$. Note that if $p>7$, then $\lambda^{-1}Y\subseteq X$ cannot contain $\{u,u+1\}$ for some $u\in\Z_p$ since $\{u,u+1\}\subseteq \lambda^{-1}Y$ implies $\lambda\in\{\pm1, \pm2\}$; in the same way, $\{0,2\}\subseteq \lambda^{-1}Y$ implies $p\leq 7$. We have the remaining cases:
\begin{itemize}
\item $\lambda^{-1}Y=\{0,\lambda^2+1,2\lambda^2+2\}$. Then $Y=\{0,\lambda^3+\lambda,2\lambda^3+2\lambda\}$ so $\lambda^3+\lambda\in\{\pm 1\}$ by Observation~\ref{rem:useful2}, and we conclude $-1=\lambda^4+\lambda^2\in\{\pm \lambda\}$ hence $p=3$.

\item $\lambda^{-1}Y=\{0,\lambda^2+2,2\lambda^2+2\}$. Then $Y=\{0,\lambda^3+2\lambda,2\lambda^3+2\lambda\}$. Then $\lambda^3\in\{\pm 1\}$ implies $\lambda=0$ which is impossible.

\item $\lambda^{-1}Y=\{1,\lambda^2+1,2\lambda^2+2\}$.  Then $Y=\{\lambda,\lambda^3+\lambda,2\lambda^3+2\lambda\}$. If $\lambda^3=1$, then $\lambda=-3$ and $p\leq 7$; in the same way, $\lambda^3=-1$ implies $p\leq 7$.

\item $\lambda^{-1}Y=\{1,\lambda^2+2,2\lambda^2+2\}$. Then $Y=\{\lambda,\lambda^3+2\lambda,2\lambda^3+2\lambda\}$. If $\lambda^3=1$, then $\lambda=-3$ and $p\leq 7$; in the same way, $\lambda^3=-1$ implies $p\leq 7$

\item $\lambda^{-1}Y=\{2,\lambda^2+1,2\lambda^2+2\}$. Then $Y=\{2\lambda,\lambda^3+\lambda,2\lambda^3+2\lambda\}$. Then $\lambda^3\in\{\pm 1\}$ implies $\lambda=0$ which is impossible or $p\leq 3$. 

\item $\lambda^{-1}Y=\{2,\lambda^2+1,2\lambda^2\}$. Then $Y=\{2\lambda,\lambda^3+2\lambda,2\lambda^3+2\lambda\}$; if $\lambda^3=1$ then $w=2\lambda$, and if $\lambda^3=-1$ then $w=2\lambda-2$. 
\end{itemize}
\end{proof}

Finally, we are ready to prove the lemma who dismiss  the existence of rainbow--free colorings with $|A|=|B|=3$.

 \begin{lemma}\label{lem:A3B3}
Every $3$--coloring $\Z_p=A\cup B \cup C$ with $|A|=|B|=3$ and $|B|\leq |C|$ contains a rainbow solution of Equation~(\ref{eq:general2}). 
\end{lemma}

\begin{proof} 
By Proposition \ref{prop:1,3b} we may assume that there are $a_i\not\in \{\pm a_j\}$, without loss of generality let $a_1\not\in \{\pm a_2\}$.  We will show that:
\begin{equation}\label{eq:last1}
|a_1A+a_2B\cup a_1B+a_2A|>6
\end{equation}
which  implies the lemma since $a_1A+a_2B\cup a_1B+a_2A$ would not be contained in $\Z_p\setminus -a_3C$. 

 If $|a_1A+a_2B|=5$, then by Theorem \ref{thm:vos}
\begin{equation*}
a_1A=\{u,u+d,u+2d\}\qquad\text{and}\qquad a_2B=\{w,w+d,w+2d\}
\end{equation*}
for some $u,w,d\in\Z_p$. Lemma \ref{lem:X3Y3} and Observation \ref{rem:useful2} imply $|a_1B+a_2A|>6$. In the same way if $|a_1B+a_2A|=5$, Equation (\ref{eq:last1}) follows.

Suppose there are not rainbow solutions of  Equation~(\ref{eq:general2}) so by the analysis done at the beginning this section and last paragraph 
\begin{equation}\label{eq:last2}
|a_1A+a_2B|=|a_1B+a_2A|=6
\end{equation}
 and  Equation (\ref{eq:last1}) needs to be  false. 
 
First assume either there is not $r\in \Z_p$ such that $a_1A=a_2B+r$ or there is not $r\in \Z_p$ such that $a_2A=a_1B+r$; Without loss of generality there is not $r\in \Z_p$ such that $a_1A=a_2B+r$ thus by Lemma \ref{lem:X3Y3} there are $u,w,d\in\Z_p$ such that
\begin{equation*}
\{a_1A, a_2B\}=\{\{u,u+d,u+2d\}, \{w,w+d,w+3d\}\}
\end{equation*} 
Lemma \ref{lem:X3Y3} and Observation \ref{rem:useful2} imply $|a_1B+a_2A|>6$ which contradicts our assumption. 
 
 Now take $r_1,r_2\in \Z_p$ such that $a_1A=a_2B+r_1$ and $a_2A=a_1B+r_2$; Write $\lambda:=a_2a_1^{-1}$ and $\mu:=a_2a_1^{-2}r_1-a_1^{-1}r_2$ so 
 \begin{equation}\label{eq:last3}
 B=\lambda^2B+\mu. 
 \end{equation}
 If $\lambda^2=-1$, $B$ is an arithmetic progression; consequently $a_2B, a_1A$ are arithmetic progressions with the same common difference contradicting Equation (\ref{eq:last2}). From now on we assume $\lambda^2\neq -1$; write $B=\{b_1,b_2,b_3\}$. We claim there is not $b'\in B$ such that $b'=\lambda^2b'+\mu$; indeed if there is a $b'$ like this, then $b''\neq \lambda^2b''+\mu$ for all $b''\in B\setminus\{b'\}$   but Equation (\ref{eq:last3}) yields
 \begin{equation*}
 b''=\lambda^2(\lambda^2b''+\mu)+\mu=\lambda^4b''+(\lambda^2+1)\mu
 \end{equation*}
 so $b''=\lambda^2b''+\mu$ since $\lambda^2\neq -1$. Thus without loss of generality
 \begin{equation*}
 b_2=\lambda^2b_1+\mu, \hspace{.4cm}b_3=\lambda^2b_2+\mu,   \hspace{.4cm} b_1=\lambda^2b_3+\mu, 
 \end{equation*}
 and particularly $1+\lambda^2+\lambda^4=0$. By Equation (\ref{eq:last2}) and since we assumed Equation (\ref{eq:last1}) is false
 \begin{equation*}
 a_1A+a_2B=a_1B+a_2A
 \end{equation*}
 so
 \begin{equation*}
 \lambda B+\lambda B=B+B-(z_1-z_2)a_1^{-1}
 \end{equation*}
 hence adding $-2b_1$ and multiplying by $\theta:=((\lambda^2-1)b_1+\mu)^{-1}$ to this equation we obtain 
 \begin{align}
 \lambda\{0,1,2,\lambda^2+1,\lambda^2+2,2\lambda^2+2\}&=\{0,1,2,\lambda^2+1,\lambda^2+2,2\lambda^2+2\}\nonumber\\
 &\quad+((r_2-r_1)a_1^{-1}+2b_1(1-\lambda))\theta.\label{eq:last4}
 \end{align}
 By Lemma \ref{lem:lambda} we have either 
 \begin{equation}\label{eq:last5}
 \lambda^3=1\qquad\text{and}\qquad 2\lambda=((r_2-r_1)a_1^{-1}+2b_1(1-\lambda))\theta
 \end{equation} 
 or 
  \begin{equation}\label{eq:last6}
 \lambda^3=-1\qquad\text{and}\qquad 2\lambda-2=((r_2-r_1)a_1^{-1}+2b_1(1-\lambda))\theta.
 \end{equation}
 If Equation (\ref{eq:last5}) holds true, then 
 \begin{equation*}
 2\lambda\mu=(r_2-r_1)a_1^{-1}
 \end{equation*}
 and thereby
  \begin{equation}\label{eq:last7}
r_2(1+2\lambda)=r_1(1+2\lambda^2);
 \end{equation}
 on the other hand 
 \begin{equation*}
 A=\lambda B+a_1^{-1}r_1
 \end{equation*}
 and  
 \begin{equation*}
 B=\lambda B+(\lambda^2+1)\mu
 \end{equation*}
 by Equation (\ref{eq:last3}) however Equation (\ref{eq:last7}) implies $A=B$ which is a contradiction.  If Equation (\ref{eq:last6}) holds true, then by Equation (\ref{eq:last4})
 \begin{equation*}
 \{0,\lambda-1,\lambda\}=\{3\lambda-2,3\lambda-1,4\lambda\}
 \end{equation*}
 which is impossible and thereby Equation (\ref{eq:last1}). 
 \end{proof}

\end{document}